\documentclass[oribibl]{llncs}
\usepackage{llncsdoc}

\usepackage{amsmath, amssymb, amscd, mathrsfs}

\newtheorem{prop}[theorem]{Proposition}
\newtheorem{cor}[theorem]{Corollary}

\newcommand{\Con}{\mathcal{C}}
\newcommand{\Vars}{\mathcal{P}}
\newcommand{\Frm}{\mathsf{Fm}}

\newcommand{\Rules}{\mathsf{R}}
\newcommand{\ruleR}{\mathsf{r}}

\newcommand{\MCon}{\mathcal{M}}

\newcommand{\bydef} {:=}                       
\newcommand{\means}{\leftrightharpoons}
\newcommand{\dimpl}{\Rightarrow}

\newcommand{\LogL}{\mathsf{L}}
\newcommand{\Int}{\mathsf{Int}}
\newcommand{\SF}{\mathsf{S4}}
\newcommand{\DP}{\mathsf{DP}}
\newcommand{\KF}{\mathsf{K4}}
\newcommand{\GL}{\mathsf{GL}}
\newcommand{\Grz}{\mathsf{Grz}}
\newcommand{\Dn}{\mathsf{D_n}}
\newcommand{\KM}{\mathsf{KM}}

\newcommand{\DS}{\mathsf{S}}

\newcommand{\one}{\mathbf{1}}
\newcommand{\zero}{\mathbf{0}}

\newcommand{\var}{\mathbb{V}}
\newcommand{\qvar}{\mathbb{Q}}
\newcommand{\univ}{\mathbb{U}}

\newcommand{\CPC}{\mathsf{CPC}}
\newcommand{\IPC}{\mathsf{IPC}}
\newcommand{\ML}{\mathsf{ML}}

\newcommand{\Heyt}{\mathcal{H}}

\newcommand{\Admm}{Adm}
\newcommand{\Adms}{Adm^{(1)}}

\newcommand{\Alg}[1]{\mathbf{#1}}
\newcommand{\class}[1]{\mathcal{#1}}
\newcommand{\set}[2]{\{#1 \mid #2\}}

\title{Multiple Conclusion Rules in Logics with the Disjunction Property}

\titlerunning{Multiple Conclusion Rules in Logics with the $\DP$}

\author{Alex Citkin}
\authorrunning{A. Citkin}

\institute{Metropolitan Telecommunications, New York}

\begin{document}

\maketitle
\begin{abstract} We prove that for the intermediate logics with the disjunction property any basis of admissible rules can be reduced to a basis of admissible m-rules (multiple-conclusion rules), and every basis of admissible m-rules can be reduced to a basis of admissible rules. These results can be generalized to a broad class of logics including positive logic and its extensions, Johansson logic, normal extensions of $\SF$, n-transitive logics and intuitionistic modal logics.
\end{abstract}

\section{Introduction}

The notion of admissible rule evolved from the notion of auxiliary rule: if in a given calculus (deductive system) $\DS$ a formula $B$ can be derived from a set of formulas $A_1,\dots,A_n$, one can shorten derivations by using a rule $A_1,\dots,A_n/B$. The application of such a rule does not extend the set of theorems, i.e. such a rule is admissible (permissible). In \cite[p.19]{Lorenzen_Book_1955} P.~Lorenzen called the rules not extending the class of the theorems "zul\"assing", and the latter term was translated as "admissible", the term we are using nowadays. In \cite{Lorenzen_Protologik_1956} Lorenzen also linked the admissibility of a rule to existence of an elimination procedure. 

Independently, P.S.~Novikov, in his lectures on mathematical logic, had introduced the notion of derived rule: a rule $A_1,\dots,A_n/B$ is derived in a calculus $\DS$ if $\vdash_\DS B$ holds every time when $\vdash_\DS A_1,\dots,\vdash_\DS A_n$ (see \cite[p. 30]{Novikov_Book}\footnote{This book was published in 1977, but it is based on the notes of a course that P.S.~Novikov taught in 1950th; A.V.~Kuznetsov was recalling that P.S.~Novikov had used the notion of derivable rule much earlier, in this lectures in 1940th.}). And he distinguished between two types of derived rules: a derived rule is strong, if $\vdash_\DS A_1 \to (A_2 \to \dots (A_n \to B) \dots)$ holds, otherwise a derived rule is weak.       

For classical propositional calculus ($\CPC$), the use of admissible rules is merely a matter of convenience, for every admissible for $\CPC$ rule $A_1,\dots,A_n/B$ is derivable, that is $A_1,\dots,A_n \vdash B$ (see, for instance \cite{Belnap_et_Strengthening_1963}). It was observed by R.~Harrop in \cite{Harrop_Concerning_1960} that the rule $\neg p \to (q \lor r)/(\neg p \to q) \lor (\neg p \to r)$ is admissible for the intuitionistic propositional calculus ($\IPC$), but is not derivable in $\IPC$. Later, in mid 1960th, A.V.~Kuznetsov observed that the rule $(\neg\neg p \to p) \to (p \lor \neg p)/((\neg\neg p \to p) \to \neg p) \lor ((\neg\neg p \to p) \to \neg\neg p)$ is also admissible for $\IPC$, but not derivable. Another example of an admissible for $\IPC$ not derivable rule was found in 1971 by G.~Mints (see \cite{Mints1971}). 

In 1974 A.V.~Kuznetsov asked whether admissible for $\IPC$ rules have a finite basis, that is, whether there is a finite set $\Rules$ of admissible for $\IPC$ rules such that every admissible for $\IPC$ rule can be derived from $\Rules$. Independently, in \cite[Problem 40]{Friedman_Hundred_1975} H.~Friedman asked whether the problem of admissibility for $\IPC$ is decidable, that is, whether there is a decision procedure that by a given rule $\ruleR$ decides whether $\ruleR$ is admissible for $\IPC$. Also, in \cite{Pogorzelski_Structural_1971} W.~Pogorzelski introduced a notion of structural completeness: as deductive system  $\DS$ is structurally complete if every admissible for $\DS$ structural rule is derivable in $\DS$. Thus, $\CPC$ is structurally complete, while $\IPC$ is not. Naturally, a question which intermediate logics are structurally complete has been posed. Thus, for intermediate logics, and, later, for modal and various types of propositional (and not only propositional) logics, for a given logic $\LogL$, first, we ask (a) whether $\LogL$ is structurally complete, that is, whether there are admissible for $\LogL$ not derivable rules; if $\LogL$ is not structurally complete, we ask (b) whether admissible for $\LogL$ rules have a finite, or at least recursive\footnote{Using idea from \cite{Craig_Axiomatizability_1953}, it is not hard to show that if a logic has a recursively enumerable explicit basis of admissible rules, it has a recursive basis.}, basis; or, at last, (c) whether a problem of admissibility for $\LogL$ is decidable\footnote{In \cite{Chagrov_Decidable_1992} A.~Chagrov has constructed a decidable modal logic having undecidable admissibility problem, and gave a negative answer to V.Rybakov's question \cite[Problem (1)]{Rybakov_Problems_1989}. The problem whether there exists a decidable intermediate logic with undecidable admissibility problem remains open.}.

It was established by V.~Rybakov (see \cite{Rybakov_Criterion_Adm_1984,Rybakov_Bases_Adm_1985}) that there is no finite basis of admissible for $\Int$ (and $\SF$) rules, i.e. Kuznetsov's question has a negative answer, but the problem of admissibility for $\Int$ (and $\SF$) is decidable, i.e. Friedman's problem has a positive answer. Later, using ideas from \cite{Rybakov_Criterion_Adm_1984,Rybakov_Bases_Adm_1985}, V.~Rybakov has constructed a basis of admissible rules for $\SF$ (see \cite{Rybakov_Construction_2001}). For $\Int$, P.~Rozi{\'e}re (see \cite{Roziere1992} and R.~Iemhoff (see \cite{Iemhoff2001}), using different techniques, have found a recursive basis of admissible rules. Using this technique, R.~Iemhoff has found the bases of admissible rules for different intermediate logics (see \cite{Iemhoff2006,Iemhoff2005}). Some very useful information on admissibility in intermediate logics as well as in modal logics can be found in the book \cite{Rybakov_Book}  by V.~Rybakov.  

In the review \cite{Kracht_Review_1999} on aforementioned book \cite{Rybakov_Book}, M.~Kracht suggested to study admissibility of multiple-conclusion rules: a rule $A_1,\dots,A_n/B_1,\dots,B_n$ is admissible for a logic $\LogL$ if every substitution that makes all the premises valid in $\LogL$, makes at least one conclusion valid in $\LogL$ (see also \cite{Kracht_Modal_2007}). A natural example of multiple-conclusion rule (called m-rule for short) admissible for $\IPC$ is the following rule, representing the disjunction property (DP for short): $\DP \bydef p \lor q/p,q$. That is, if a formula $A \lor B$ is valid in $\IPC$, then at least one of the formulas $A,B$ is valid in $\IPC$ (for more on disjunction property see \cite{Chagrov_Zakh_1991}). It was reasonable to ask the same questions regarding m-rules: whether a given logic has admissible, not derived m-rules, whether m-rules have a finite or recursively enumerable basis, or whether the admissibility of m-rules is decidable. The bases of m-rules for a variety of intermediate and normal modal logics were constructed in \cite{Jerabek_Independent_2008,Jerabek_Canonical_2009,Goudsmit_Iemhoff_Unification_2014,Goudsmit_Note_2013,Goudsmit_PhD}.
 
For logics with the DP, there is a close relation between m-rules and rules: with each m-rule $\ruleR \bydef \Gamma/\Delta$ one can associate a rule $\ruleR^q \bydef \bigwedge\Gamma \lor q/\bigvee\Delta \lor q$, where variable $q$ does not occur in formulas from $\Gamma,\Delta$. Our goal is to prove that if m-rules $\ruleR_i, i\in I$ form a basis of m-rules admissible for a given intermediate logic $\LogL$ with the DP, then rules $\ruleR^q, i \in I$ form a basis of rules admissible for $\LogL$ (comp. \cite[Theorem 3.1]{Jerabek_Independent_2008}). To prove this, we will use the main theorem from \cite{Citkin_2011}. As a consequence, we obtain that for intermediate logics with the DP, each of the mentioned above problems for the m-rules and rules are equivalent. In the last section, we will discuss how this result can be extended beyond intermediate logics. In order to extend the results from intermediate logics to normal extension of $\SF$, we are not using G\"odel-McKinsey-Tarski translation; instead, we make a use of some common properties of the algebraic models (Heyting algebras and $\SF$-algebras), and this gives us an ability to extend the results even further.

\section{Background.}

\subsection{Multiple-Conclusion Rules.}
We consider (propositional) formulas built in a usual way from the propositional variables from a countable set $\Vars$ and connectives from a finite set $\Con$. By $\Frm$ we denote the set of all formulas, and by $\Sigma$ we denote the set of all substitutions, that is the set of all mappings $\sigma:\Vars \to \Frm$. In a natural way, every substitution $\sigma$ can be extended to a mapping $\Frm \to \Frm$.

A \textit{multiple-conclusion rule} (m-rule for short) is an ordered pair of finite sets of formulas $\Gamma,\Delta \subseteq \Frm$ written as  $\Gamma/\Delta$;  $\Gamma$ is a set of \textit{premises}, and $\Delta$ is a set of \textit{conclusions}. A \textit{rule} is an m-rule, that has the set of conclusions consisting of a single formula. 

A \textit{structural multiple-conclusion consequence relation} (m-consequence for short) is a binary relation $\vdash$ between finite sets of formulas for which the following holds: for any formula $A \in \Frm$ and any finite sets of formulas $\Gamma,\Gamma',\Delta,\Delta' \subseteq \Frm$
\begin{itemize}
\item[(R)] $A \vdash A$;  
\item[(M)] if $\Gamma \vdash \Delta$, then $\Gamma \cup \Gamma' \vdash \Delta \cup \Delta'$; 
\item[(T)] if $\Gamma,A \vdash\Delta$ and $\Gamma'\vdash A,\Delta'$, then $\Gamma \cup \Gamma' \vdash \Delta \cup \Delta'$;
\item[(S)] if $\Gamma \vdash \Delta$, then $\sigma(\Gamma) \vdash \sigma(\Delta)$ for each substitution $\sigma \in \Sigma$. 
\end{itemize}
A class of all m-consequences will be denoted by $\MCon$.

Let $\vdash$ be an m-consequence  and $\ruleR \bydef \Gamma/\Delta$ be an m-rule. An m-rule $\ruleR $ is \textit{derivable} w.r.t. $\vdash$ (in written $\vdash \ruleR$), if $\Gamma \vdash \Delta$.

Every collection $\Rules$ of m-rules defines an m-consequence $\vdash_\Rules$, namely, the least m-consequence relative to which every rule from $\Rules$ is derivable:
\[
\vdash_\Rules \bydef \bigcap\set{\vdash \ \in \MCon}{\vdash \ruleR \text{ for every } \ruleR \in \Rules}.
\]  

An m-rule $\ruleR$ is said to be \textit{derivable from a set of m-rules} $\Rules$ (in written $\Rules \vdash \ruleR$), if $\vdash_\Rules \ruleR$.  

Every m-consequence $\vdash$ defines a logic $\LogL(\vdash) \means \set{A \in \Frm}{\vdash A}$. If $\LogL$ is a logic, an m-rule $\Gamma/\Delta$ is said to be \textit{admissible} for $\LogL$ if for every substitution $\sigma \in \Sigma$
\[
\sigma(\Gamma) \subseteq \LogL \text{ entails } \sigma(\Delta) \cap \LogL \neq \emptyset. 
\]
If $\LogL$ is a logic, by $\Admm(\LogL)$ we denote a set of all m-rules admissible for $\LogL$, and by $\Adms(\LogL)$ we denote a set of all rules admissible for $\LogL$.

Given a logic $\LogL$, a set of m-rules $\Rules$ forms an \textit{basis of admissible m-rules} (m-basis for short), if every rule $\ruleR \in \Admm(\LogL)$ is derivable from $\Rules$; and a set of rules $\Rules$ forms a \textit{basis of admissible rules} (s-basis for short), if every rule $\ruleR \in \Adms(\LogL)$ is derivable from $\Rules$.  

\subsection{Algebraic Semantics.}

\subsubsection{Basic Definitions.}

Algebraic models for intermediate logics are Heyting algebras, that is algebras $\langle \mathsf{A}; \land,\lor,\to,\neg,\one,\zero \rangle$, where $\langle \mathsf{A}; \land,\lor,\one,\zero \rangle$ is a bounded distributive lattice, and $\to,\neg$ are respectively a relative pseudo-complement and a pseudo-complement. 

Let $\Alg{A}$ be a (Heyting) algebra, $A$ be a formula, $\ruleR' \bydef A_1,\dots,A_n/B$ be a rule and $\ruleR \bydef A_1,\dots,A_n/B_1,\dots,B_m$ be an m-rule. \textit{A formula $A$ is valid in a (Heyting) algebra} $\Alg{A}$ (in written, $\Alg{A} \models A$) if for every assignment $\nu: \Vars \to \Alg{A}$ the value $\nu(A)$, that is, the value obtained by interpreting the connectives by operations of $\Alg{A}$, is $\one$. Accordingly, \textit{rule $\ruleR'$ is valid in }$\Alg{A}$ (in written, $\Alg{A} \models \ruleR'$), if for every assignment $\nu$, if $\nu(A_1) = \dots = \nu(A_n) = \one$ yields $\nu(B) = \one$. And m-rule $\ruleR$ if for every assignment $\nu$, $\nu(A_1) = \dots = \nu(A_n) = \one$ yields that at least for some $j =1,\dots,m$, $\nu(B_j) =  \one$. 

Let $\class{K}$ be a set of algebras. If $F$  is a family of formulas ($\Rules'$ is a family of rules, or $\Rules$ is a family of m-rules), then by $\class{K} \models F$ ($\class{K} \models \Rules'$ or $\class{K} \models \Rules$) we mean that every formula (rule of m-rule) is valid in each algebra $\Alg{A} \in \class{K}$.

Immediately from the definition of validity of rule, we have the following:
\begin{prop} \label{propprod} Let $\ruleR$ be a rule and $\Alg{A}_i,i \in I$ be a family of algebras. Then, $\Alg{A}_i \models A$ for all $i \in I$ if and only if $\prod_{i\in I}\Alg{A}_i \models \ruleR$. 
\end{prop}

Let us observe that for m-rules the situation is quite different: if $\Alg{A}$ is a two-element Boolean algebra, then $\Alg{A} \models \DP$, but $\Alg{A}^2 \not\models \DP$.

It is not hard to see that any set of formulas $F$ defines a variety $\var(F) \means \set{\Alg{A}}{\Alg{A} \models F}$; any set $\Rules'$ of rules defines a quasivariety $\qvar(\Rules) = \set{\Alg{A}}{\Alg{A} \models \Rules'}$; any set $\Rules$ of m-rules defines a universal class $\univ(\Rules) = \set{\Alg{A}}{\Alg{A} \models \Rules'}$). 

On the other hand, if $\class{K}$ is a family of algebras, by $\var(\class{K})$, $\qvar(\class{K})$ and $\univ(\class{K})$, we denote respectively a variety, quasivariety and universal class generated by algebras $\class{K}$.

There is 1-1-correspondence between intermediate logics and non-trivial varieties of Heyting algebras. Moreover, there is 1-1correspondence between consequence relations and subquasivarieties of $\Heyt$, and between m-consequences and universal subclasses of $\Heyt$ (see, for instance, \cite{Cabrer_Metcalfe_Admissibility_2015}). If $\class{V}$ is a variety corresponding to a logic $\LogL$, then a formula $A$ is valid in $\LogL$ (a rule $\ruleR'$ is admissible for $\LogL$, or an m-rule $\ruleR$ is m-admissible for $\LogL$) if and only if $\Alg{F}_\class{V} \models A$ (accordingly $\Alg{F}_\class{V} \models \ruleR'$, or $\Alg{F}_\class{V} \models \ruleR$.  

Let us note the following important property: if $\Rules$ is a set of m-rules (or rules) and $\ruleR$ is an m-rule (or a rule), then $\Rules \vdash \ruleR$ if and only if 
\begin{equation}
\Alg{A} \models \Rules \text{ entails } \Alg{A} \models \ruleR \text{ for every algebra } \Alg{A} \in \Heyt. \label{algrulent}
\end{equation}

\subsubsection{Well-connected Algebras.}

An algebra $\Alg{A}$ is called \textit{well-connected}, if for every $\Alg{a},\Alg{b} \in \Alg{A}$, if $\Alg{a} \lor \Alg{b} = \one$, then $\Alg{a} = \one$ or $\Alg{b} = \one$.

The class of all Heyting algebras forms a variety $\Heyt$, and the free algebras of $\Heyt$ are well-connected.

\begin{prop}\label{seqprop} Let $\Alg{A}$ be a well-connected algebra, $\Gamma/\Delta$ be an m-rule and $q$ be a variable not occurring in $\Gamma/\Delta$. Then the following is equivalent
\begin{itemize}
\item[(a)] $\Alg{A} \models \Gamma/\Delta$;
\item[(b)] $\Alg{A} \models \Gamma/\bigvee_{B \in \Delta}B$;
\item[(c)] $\Alg{A} \models \bigwedge_{A \in \Gamma}A \lor q/\bigvee_{B \in \Delta}B \lor q$;
\end{itemize} 
\end{prop}
\begin{proof}
(a) $\dimpl$ (b) is trivial. 

(b) $\dimpl$ (a) due to well-connectedness of $\Alg{A}$.

(b) $\dimpl$ (c). Suppose $\Alg{A} \not\models \bigwedge_{A \in \Gamma}A \lor q/\bigvee_{B \in \Delta}B \lor q$. We need to prove that $\Alg{A} \not\models \Gamma/\bigvee_{B \in \Delta}B$. 

Indeed, let $\nu$ be a refuting valuation, that is
\begin{equation}
\bigwedge_{A \in \Gamma}\nu(A) \lor \nu(q) = \one_\Alg{A} \text{ while } \bigvee_{B \in \Delta}\nu(B)  \lor \nu(q) \neq \one_\Alg{A}. \label{seqprop1}
\end{equation}
Then, clearly,
\begin{equation}
\bigvee_{B \in \Delta}\nu(B) \neq \one_\Alg{A} \label{seqprop2}
\end{equation}
and
\begin{equation}
\nu(q) \neq \one_\Alg{A}. \label{seqprop3}
\end{equation}
Due to well-connectedness of $\Alg{A}$, from \eqref{seqprop1} and \eqref{seqprop3} we have
\begin{equation}
\bigwedge_{A \in \Gamma}\nu(A) = \one_\Alg{A} \label{seqprop4}.
\end{equation}
And \eqref{seqprop4} together with \eqref{seqprop2} mean that $\nu$ is a refuting valuation from $\Gamma/\bigvee_{B \in \Delta} B$, that is,  $\Alg{A} \not\models \Gamma/\bigvee_{B \in \Delta} B$.

(c) $\dimpl$ (b). Since $q$ does not occur in the formulas from $\Gamma,\Delta$, we can substitule $q$ with $\zero_\Alg{A}$  and reduce (c) to (b).
\end{proof}

The above Proposition can be restated in the following way:

\begin{cor} \label{cseqprop} Let $\Alg{A}$ be a well-connected algebra, $\ruleR$ be an m-rule and $q$ be a variable not occurring in $\ruleR$. Then $\Alg{A} \models \ruleR$ if and only if $\Alg{A} \models \ruleR^q$.
\end{cor}
\section{The Case of Intermediate Logics.}

In this section we prove that for the intermediate logics with the disjunction property, any basis of admissible rules can be reduced to a basis of admissible m-rules (multiple-conclusion rules), and every basis of admissible m-rules can be reduced to a basis of admissible rules. 

\subsection{Reductions.} \label{redadm}

We consider formulas in the signature $\land, \lor, \to, \neg, \bot, \top$. \textit{Intermediate logic} is understood as a set of formulas $\LogL$ such that $\Int \subseteq \LogL \subset \Frm$ and closed under Modus Ponens. By $\vdash_\Int$ we denote a consequences relation defined by intuitionistic axiom schemata and the rule Modus Ponens. In this section we consider only m-consequences extending  $\vdash_\Int$ and defining intermediate logics. Clearly, for each intermediate logic $\LogL$ there is an m-consequence defining it: one can take a consequence relation that is defined by $\LogL$ (viewed as a set of axiom schemata) and by Modus Ponens. 

A (intermediate) logic $\LogL$ enjoys the \textit{disjunction property} (DP for short) if $ (A\lor B) \in \LogL$ yields $A \in \LogL$ or $B \in \LogL$ for any formulas $A,B$. It is clear that $\LogL$ has the DP if and only if m-rule 
\[
\DP \bydef p \lor q/p,q
\] is admissible for $\LogL$.

\begin{definition} Let $\ruleR \bydef \Gamma/\Delta$ be an m-rule. The following rule is called a \textit{reduction of rule} $\ruleR$:
\begin{equation} 
\ruleR^\circ \means \bigwedge_{A \in \Gamma}A/\bigvee_{B \in \Delta}B, \label{reduction}
\end{equation}
where $\bigwedge_{A \in \Gamma}A = \top$, if $\Gamma = \emptyset$, and $\bigvee_{B \in \Delta}B = \bot$, if $\Delta = \emptyset$.
\end{definition}
Note, that the rule $\ruleR^\circ$ is always a single-conclusion rule.

As we know, rule $\DP$ expresses the DP and the following holds: 

\begin{prop} \label{propdp} Let $\Rules$ be a set of rules from which $\DP$ is derived, and $\ruleR \bydef \Gamma/\Delta$ be an m-rule. Then
\begin{equation}
\Rules \vdash \ruleR \text{ if and only if } \Rules \vdash \ruleR^\circ. \label{propdp0}
\end{equation}
\end{prop}
\begin{proof} ($\Rightarrow$) Suppose $\Rules \vdash \ruleR$ , that is, $\Gamma \vdash_\Rules \Delta$. We need to prove $\Rules \vdash \ruleR^\circ$, that is, we need to show that $\bigwedge_{A \in \Gamma}A \vdash_\Rules \bigvee_{B \in \Delta} B$.

If $\Delta = \emptyset$, then $\Gamma \vdash_\Rules \emptyset$ yields $\Gamma \vdash_\Rules \bot$, because $\vdash_\Rules$ is closed under (M). In its turn, $\Gamma \vdash_\Rules \bot$ entails $\bigwedge_{A \in \Gamma}A \vdash_\Rules \bot$, for $\Gamma \vdash_\Int \bigwedge_{A \in \Gamma}A$, and $\vdash_\Int \ \subseteq \ \vdash_\Rules$. 

The case $\Delta = \{B\}$ is trivial.

Suppose $\Delta = \{B_1,\dots,B_n,B_{n+1}\}$. Let us prove that
\begin{equation}
\Gamma \vdash_\Rules B_1,B_2,\Delta' \text{ yields } \Gamma \vdash_\Rules B_1 \lor B_2,\Delta' \label{propdp1} 
\end{equation} 
and then one can complete the proof  of ($\Rightarrow$) by induction on cardinality of $\Delta$.

Assume
\begin{equation}
\Gamma \vdash_\Rules B_1,B_2,\Delta'. \label{propdp2} 
\end{equation} 
Let us observe that $B_1 \vdash_\Int B_1 \lor B_2$, and, by (M), we have $B_1 \vdash_\Int B_1 \lor B_2, B_2,\Delta'$. Since $\vdash_\Int \  \subseteq \ \vdash_\Rules$, we can conclude that
\begin{equation}
B_1 \vdash_\Rules B_1 \lor B_2, B_2, \Delta'. \label{propdp3}
\end{equation}
From \eqref{propdp2} and \eqref{propdp3} by (T) we have  
\begin{equation}
\Gamma \vdash_\Rules B_1 \lor B_2, B_2,\Delta'. \label{propdp4}
\end{equation}
Now, we use $B_2 \vdash_\Int B_1 \lor B_2$, and by (M) and $\vdash_\Int \ \subseteq \ \vdash_\Rules$ we get
\begin{equation}
B_2 \vdash_\Rules B_1 \lor B_2, \Delta'. \label{propdp5}      
\end{equation}
And from \eqref{propdp4}  and \eqref{propdp5} by (T) we obtain 
\begin{equation}
\Gamma \vdash_\Rules B_1 \lor B_2, \Delta',
\end{equation} 
and this completes the proof of $\dimpl$.

Proof of ($\Leftarrow$). Suppose $R \vdash r^\circ$, i.e. $\bigwedge_{A \in \Gamma}A \vdash_\Rules \bigvee_{B \in \Delta}B$. Then, due to $\Gamma \vdash_\Int \bigwedge_{A \in \Gamma}A$, we get $\Gamma \vdash_\Rules \bigvee_{B \in \Delta}B$. And, since $\Rules \vdash \DP$, we have $\bigvee_{B \in \Delta}B \vdash_\Rules \Delta$. Thus, $\Gamma \vdash_\Rules \Delta$, that is, $\Rules \vdash \ruleR$.
\end{proof}

\subsection{$q$-Reductions.} \label{qrecuct}

\begin{definition}
With every m-rule $\ruleR \bydef \Gamma/\Delta$ and a variable $q$ we associate a rule
\begin{equation}
\ruleR^q \bydef \bigwedge_{A \in \Gamma}A \lor q/\bigvee_{B\in \Delta}B \lor q. \label{qreduction}
\end{equation} 
\end{definition}
The rule $\ruleR^q$ we call a $q$\textit{-reduction} of the rule $\ruleR$. If $\Rules$ is a set of m-rules and $q$ is a variable, we let $\Rules^q \means \set{r^q}{r \in \Rules}$.

\begin{prop} \label{admsubs} If an m-rule $\Gamma/\Delta$ is admissible for a given logic $\LogL$, then for every substitution $\sigma \in \Sigma$ the m-rule $\sigma(\Gamma)/\sigma(\Delta)$ is admissible for $\LogL$.
\end{prop}
\begin{proof} The proof follows immediately from the definition of admissible m-rule and from the observation that a composition of two substitutions is a substitution.
\end{proof}

\begin{prop} \label{admprop} Let a logic $\LogL$ enjoys DP and $q$ be a variable not occurring in an m-rule $\ruleR$. Then m-rule $\ruleR$ is admissible for $\LogL$ if and only if the rule $\ruleR^q$ is admissible for $\LogL$.
\end{prop}

\begin{proof} Let $\ruleR \bydef \Gamma/\Delta$ be admissible for $\LogL$. We need to prove that for every substitution $\sigma \in \Sigma$, 
\begin{equation}
\text{if } \sigma(\bigwedge_{A \in \Gamma}A \lor q) \in \LogL \text{ then } \sigma(\bigvee_{B \in \Delta}B \lor q) \in \LogL. \label{admprop1}
\end{equation}

Indeed, if $\sigma(\bigwedge_{A \in \Gamma}A \lor q) \in \LogL$, by DP, one of the following holds 
\begin{itemize}
\item[(a)] $\sigma(\bigwedge_{A \in \Gamma}) \in \LogL$;
\item[(b)] $\sigma(q) \in \LogL$. 
\end{itemize}
In the case (b), $\sigma(q) \in \LogL$ and, clearly, $\sigma(\bigvee_{B \in \Delta}B \lor q) = \sigma(\bigvee_{B \in \Delta}B) \lor \sigma(q)) \in \LogL$.

In the case (a), $\sigma(\bigwedge_{A \in \Gamma}) \in \LogL$, hence, due to $\ruleR$ is admissible for $\LogL$, we have that $\sigma(B) \in \LogL$ for some $B \in \Delta$ and, hence, $\sigma(\bigwedge_{B \in \Delta}B) = \bigwedge_{B \in \Delta} \sigma(B) \in \LogL$. Therefore $\sigma(\bigwedge_{B \in \Delta}B  \lor q) \in \LogL$. 

Conversely, suppose that $\ruleR^q$ is admissible for $\LogL$. Recall that the variable $q$ is not occurring in $\Gamma , \Delta$, and let $\psi$ be a substitution such that $\psi: q \mapsto \bot$ and $\psi: p \mapsto p$ for all variables $p \neq q$. By virtue of Proposition \ref{admsubs}, the following rule, obtained from $\ruleR^q$ by applying $\psi$,
\begin{equation}   
\bigwedge_{A\in \Gamma}A \lor \bot/\bigvee_{B \in \Delta}B \lor \bot. \label{admprop2}
\end{equation}
is admissible for $\LogL$.

Assume that $\sigma$ is such a substitution that $\sigma(A) \in \LogL$ for all $A \in \Gamma$. Then, $\sigma(\bigwedge_{A \in \Gamma}A \lor \bot) \in \LogL$, and, due to rule \eqref{admprop2} is admissible for $\LogL$, we have
\[
\sigma(\bigvee_{B \in \Delta}B \lor \bot) \in \LogL.
\] 
Since $\sigma(\bigvee_{B \in \Delta}B \lor \bot) = \bigvee_{B \in \Delta}\sigma(B) \lor \bot$ and the right hand formula is equivalent in $\Int$ to $\bigvee_{B \in \Delta}\sigma(B)$, we have
\[
\bigvee_{B \in \Delta}\sigma(B) \in \LogL.
\] 
Due to logic $\LogL$ enjoys DP, for one of the formulas $B \in \Delta$ we have $\sigma(B) \in \LogL$, and this, by the definition of admissibility, means that the rule $\ruleR$ is admissible for $\LogL$. 
\end{proof}

\begin{cor} \label{admsubscor1} Let $\LogL$ be a logic with the DP. Then if a rule $\Gamma/\Delta$ is admissible for $\LogL$, so is the rule $\bigwedge_{A \in \Gamma}A/\bigvee_{B \in \Delta}B $.
\end{cor}
\begin{proof} Directly from the propositions \ref{admsubs} and \ref{admprop}. \end{proof}

\textbf{Note.} It is not hard to see that one can prove Corollary \ref{admsubscor1} without restriction that $\LogL$ enjoys DP. 

A \textit{problem of m-admissibility} (of admissibility) for a logic $\LogL$ is a problem of recognizing by a given m-rule (by a given rule) $\ruleR$ whether $\ruleR$ is admissible for $\LogL$, i.e. whether $\ruleR \in Adm(\LogL)$ (respectively, whether $\ruleR \in Adm^1(\LogL)$). Thus, the problem of m-admissibility (of admissibility) for $\LogL$ is decidable if and only if the set $Adm(\LogL)$ (the set $Adm^1(\LogL)$) is recursive. Recall that two decision problems are \textit{equivalent}, if they are reducible to each other.

Since $Adm^1(\LogL) \subseteq Adm(\LogL)$ for every  $\LogL$ and for every m-rule $\ruleR$ we can effectively recognize whether $\ruleR$ has a single conclusion, or not, that is, we can effectively recognize whether $\ruleR \in Adm^1(\LogL)$, the decidability of the problem of m-admissibility yields the decidability of the problem of admissibility. In case when $\LogL$ enjoys the DP, the converse also holds.   

\begin{cor} \label{admsubscor2} For every logic $\LogL$ enjoying DP, the problems of m-admissibility and admissibility are equivalent. In other words, the set $Adm(\LogL)$ is recursive if and only if the set $Adm^1(\LogL)$ is recursive.
\end{cor}

For instance, it is well known that $\Int$ enjoys the DP, hence from decidability of the admissibility of rules for $\Int$ (see \cite{Rybakov_Criterion_Adm_1984}) it follows that the problem of m-admissibility for $\Int$ is decidable (in algebraic terms, that the universal theory of the free Heyting algebras is decidable \cite[Theorem 10]{Rybakov_Criterion_Adm_1984}) .\\

\begin{remark} It is known from \cite{Prucnal_Structural_1976} that Medvedev's Logic $\ML$ is structurally complete and enjoys DP. From Proposition \ref{admprop} it immediately follows that the rule $\DP$ forms an m-basis of $\ML$. It is not hard to see that  m-rule $\DP$ is not derivable in $\ML$. In fact, for any intermediate logic $\LogL$ m-rule $\DP$ is not derivable from rules admissible for $\LogL$: all rules admissible for $\LogL$ are valid in the four-element Boolean algebra, while m-rule $\DP$ is not. 
\end{remark}

\subsection{Reduction of basis}

In Section \ref{redadm} we saw that for the logics with the DP, the admissibility of m-rule and its reduction are equivalent. In this section we will prove that the m-rules and their $q$-reductions are related even closer. More precisely, we will prove that using any basis of m-rules, one can effectively construct a basis of rules, and, using any basis of rules, one can construct a basis of m-rules.

\begin{theorem} \label{mthm} Let $\LogL$ be a logic enjoying DP. Then the following holds
\begin{itemize}
\item[(a)] If rules $\Rules$ form an s-basis, then m-rules $\Rules \cup  \{\DP\}$ form an m-basis. 
\item[(b)] If a set of m-rules $\Rules$ forms an m-basis and $q$ is a variable not occurring in any rule from $\Rules$, then the rules $\Rules^q$ form an s-basis.
\end{itemize}
\end{theorem}

\subsection*{Proof of (a)} Suppose $\Rules$ is a basis for $\LogL$. We need to prove that every rule $\Gamma/\Delta \in Adm(\LogL)$ is derived from $\Rules \cup \DP$. So, we need to prove that for every admissible m-rule $\Gamma/\Delta$ we have $\Gamma \vdash_{\Rules \cup \DP} \Delta$.

Indeed, if $\ruleR \bydef \Gamma/\Delta$ is an admissible m-rule, by Corollary \ref{admsubscor1}, the rule $\ruleR^\circ = \bigwedge_{A \in \Gamma}A/\bigvee_{B \in \Delta}B$ is admissible for $\LogL$. By our assumption,  $\Rules$ is a basis, hence, $\Rules \vdash \ruleR^\circ$, that is, 
\begin{equation}
\bigwedge_{A \in \Gamma}A \vdash_\Rules \bigvee_{B \in \Delta}B. \label{mthm1_1}
\end{equation}

Note, that the following holds for $\vdash_\Int$ 
\begin{equation}
 \Gamma \vdash_\Int \bigwedge_{A \in \Gamma}A. \label{mthm1_2}
\end{equation}
Due to $\vdash_\Int \subseteq \vdash_\Rules$, from \eqref{mthm1_2} we have
\begin{equation}
 \Gamma \vdash_\Rules \bigwedge_{A \in \Gamma}A. \label{mthm1_3}
\end{equation}
Recall, that $\vdash_\Rules$ is closed under (T), hence, from \eqref{mthm1_3} and \eqref{mthm1_1} we have
\begin{equation}
 \Gamma \vdash_\Rules \bigvee_{B \in \Delta}B, \label{mthm1_4}
\end{equation}
and, therefore,
\begin{equation}
 \Gamma \vdash_{\Rules \cup \DP} \bigvee_{B \in \Delta}B. \label{mthm1_5}
\end{equation}
Next, we apply Proposition \ref{propdp} and we obtain
\begin{equation}
 \Gamma \vdash_{\Rules \cup \DP} \Delta,
\end{equation}
i.e. $\Rules \cup \DP$ forms a basis.\\

\subsection*{Proof of (b)} Suppose $\Rules$ is a basis of admissible m-rules and $q$ is a variable not occurring in the rules from $\Rules$. We need to prove that the set $\Rules^q$ forms a basis of admissible rules. For this, we will demonstrate that $\class{Q}= \qvar(\Alg{F})$, where $\class{Q} \bydef \qvar(\Rules^q)$ and $\Alg{F}$ is a free algebra of countable rank of a variety $\var(\LogL)$.

Let $\Alg{F}$ be a free algebra of $\var(\LogL)$. Since $\LogL$ enjoys DP, $\Alg{F}$ is well-connected. Due to rules $\Rules$ are admissible for $\LogL$, the rules from $\Rules$ are valid in $\Alg{F}$. Hence, by Proposition \ref{seqprop}, all rules from $\Rules^q$ are valid in $\Alg{F}$, that is, $\Alg{F} \in \class{Q}$. Therefore, $\qvar(\Alg{F}) \subseteq \class{Q} $, and we need only to prove that $\qvar(\Alg{F}) \supseteq \class{Q}$.

For contradiction: assume that  $\qvar(\Alg{F}) \subset \class{Q}$. Then there is an algebra $\Alg{A} \in \class{Q} \setminus \qvar(\Alg{F})$ in which all rules from $\Rules^q$ are valid. By virtue of \cite[Theorem 1]{Citkin_2011}, the quasivariety $\class{Q}$ is generated by its well-connected members. Thus, we can assume that $\Alg{A}$ is well-connected. So, $\Alg{A}$ is a well-connected algebra in which all rules from $\Rules^q$ are valid. Hence, by Proposition \ref{seqprop}, all m-rules from $\Rules$ are valid in $\Alg{A}$, hence, $\Alg{A} \in \class{U}$, where $\class{U} = \univ(\Rules)$ is a universal class defined by all rules from $\Rules$. Recall, that $\Rules$ forms an m-basis and, therefore, $\class{U} = \univ(\Alg{F}) \subseteq \qvar(\Alg{F})$. Thus,
\[
\Alg{A} \in \class{U} \subseteq \qvar(\Alg{F}),
\]
and this contradicts that $\Alg{A} \in \class{Q} \setminus \qvar(\Alg{F})$.

\begin{cor} Let $\LogL$ be a logic with the DP. Then $\LogL$ has a finite (recursive, recursively enumerable) s-basis if and only if $\LogL$ has a finite (recursive, recursively enumerable) m-basis.
\end{cor}

For example, since $\Int$ does not have a finite basis of admissible rules (see \cite[Corollary 2]{Rybakov_Bases_Adm_1985}), $\Int$ does not have a finite basis of admissible m-rules too \cite[Theorem 9]{Rybakov_Bases_Adm_1985}.

\begin{cor} If $\LogL$ is a logic with the DP and $\Rules$ is an s-basis, then $\Rules^q$ is an s-basis too. In other words, every intermediate logic with the DP has an s-basis consisting of $q$-extended rules.
\end{cor}

The bases consisting of $q$-reductions of rules also have the following important property.

\begin{theorem} \label{thmind} Let $\LogL$ be a logic with the DP. If $\Rules^q$ is an independent s-basis, then $\Rules^q \cup \DP$ is an independent m-basis. 
\end{theorem}
\begin{proof}
Assume that $\Rules^q$ is an independent basis. First, we will prove that $\Rules^q \nvdash \DP$. Indeed, since $\LogL$ is an intermediate logic and, therefore, $\LogL$ is consistent, the corresponding variety $\class{V} \bydef \var(\LogL)$ is not trivial. Hence, its free algebra $\Alg{F}_\class{V}$ is not degenerate. Since all rules from $\Rules^q$ are admissible for $\LogL$, we have $\Alg{F}_\class{V} \models \Rules$. Therefore, by Proposition \ref{propprod}, we get $\Alg{F}_\class{V}^2 \models \Rules^q$. But $\Alg{F}_\class{V}^2 \not\models \DP$.

Now, let us assume that $\ruleR^q \in \Rules^q$. We need to prove that $\Rules_0^q \cup \DP \nvdash \ruleR^q$, where $\Rules_0^q \bydef \Rules^q \setminus \{\ruleR^q\}$. Let us recall that basis $\Rules^q$ is independent, that is, $\Rules_0^q \nvdash \ruleR^q$. Hence, there is an algebra $\Alg{A}$ such that $\Alg{A} \models \Rules_0^q$ and $\Alg{A} \not\models \ruleR^q$. By \cite[Lemma 1]{Citkin_2011}, $\Alg{A}$ is a subdirect product of well-connected algebras $\Alg{A}_i,i \in I$ in which all rules $\Rules_0$ are valid.  Let $\class{A} \bydef \{\Alg{A}_i,i \in I\}$. Due to all algebras from $\class{A}$ being well-connected, $\class{A} \models \Rules_0$ yields $\class{A} \models \Rules_0^q$. Since $\Alg{A} \not\models \ruleR^q$, there is an algebra $\Alg{A}_j \in \class{A}$ such that $\Alg{A}_j \not\models \ruleR^q$. Now, let us observe that the rule $\DP$ is valid in every well-connected algebra, hence $\Alg{A}_j \models \Rules_0^q \cup \DP$, but $\Alg{A}_j \not\models \ruleR^q$. And this completes the proof of the theorem.
 \end{proof}

\begin{example} The m-bases for Gabbay-de Jongh logics $\Dn$ have been constructed in \cite{Goudsmit_Iemhoff_Unification_2014}: the m-rules $\mathsf{J}_i, i \leq n+1$ (see \cite[Definition 17]{Goudsmit_Iemhoff_Unification_2014}) form a basis of m-rules of $\Dn$ for all $n$. By Theorem \ref{mthm}, $\mathsf{J}_j^q, j \leq n+1$ is a basis of admissible rules of logic $\Dn$ for all $n$.
\end{example}

\section{Beyond Intermediate Logics.}

Let us note that all proofs are based either on general properties of quasivarieties and universal classes, or on the results from \cite{Citkin_2011}. It was observed in \cite[Section 4]{Citkin_2011} that all results from \cite{Citkin_2011} can be extended to the logics for which there is a formula $R(p)$ such that $R(A) \lor R(B) \in \LogL$ yields $R(A) \in \LogL$ or $R(B) \in \LogL$, that is to the logics enjoying the DP relative to some formula $R(p)$. In this case, corresponding algebraic model $\Alg{A}$ is called the well-connected if $R(\Alg{a}) \lor R(\Alg{b}) = \one_\Alg{A}$ entails $R(\Alg{a} = \one_\Alg{A})$ or $R(\Alg{a} = \one_\Alg{A})$.

Thus, Theorems \ref{mthm} and \ref{thmind} hold for the following classes of logics

\begin{enumerate}
\item positive logic and its extensions (regarding admissibility for positive and Johansson logoics see \cite{Odintsov_Rybakov_Unification_2013});
\item minimal (Johansson) \cite{Kleene_Intro} logic and its extensions;
\item logic $\KM$ (see \cite{Muravitsky_Logic_2014}) and its extensions 
\item $\KF$ and its normal extensions;
\item intuitionistic modal logic MIPC (e.g. \cite{G_Bez_MIPC_1}) and its normal extensions;
\item n-transitive logics (e.g. \cite{Chagrov_Zakh}).
\end{enumerate} 

For instance, for logics $\KF, \SF, \Grz$ or $\GL$ one can take the m-basis (see \cite{Jerabek_Admissible_2005}) and convert it into a basis of rules (see \cite[Theorem 6.4.]{Jerabek_Admissible_2005} where the same reduction as in Theorem \ref{mthm} was used). Or one can take a basis of admissible rules of $\SF$ (see \cite{Rybakov_Construction_2001}), and convert it into am m-basis. Let us note that the proofs in \cite{Jerabek_Admissible_2005} and \cite{Rybakov_Construction_2001} are based on certain properties of Kripke models. On the other hand, an m-basis for logic $\GL$ can be obtained simply by extending the s-basis constructed in \cite{Fedorishin_Explicit_2007} by m-rule $\Box_0 p \lor \Box_0 q  /\Box_0 p ,\Box_0 q$, where $\Box_0 \alpha \means  \Box \alpha \land \alpha$. Taking into account that $\GL$ does not have finite s-basis (see \cite[Theorem 17]{Rybakov_Admissibility_Equations_1990}), we can conclude that $\GL$ has no finite m-basis.



\end{document}